\makeatletter

\def\eqalign#1{\,\vcenter{\openup\jot\m@th
  \ialign{\strut\hfil$\displaystyle{##}$&$\displaystyle{{}##}$\hfil
      \crcr#1\crcr}}\,}
\def\eqalignno#1{\displ@y \tabskip\@centering
  \halign to\displaywidth{\hfil$\displaystyle{##}$\tabskip\z@skip
    &$\displaystyle{{}##}$\hfil\tabskip\@centering
    &\llap{$##$}\tabskip\z@skip\crcr
    #1\crcr}}
\def\leqalignno#1{\displ@y \tabskip\@centering
  \halign to\displaywidth{\hfil$\displaystyle{##}$\tabskip\z@skip
    &$\displaystyle{{}##}$\hfil\tabskip\@centering
    &\kern-\displaywidth\rlap{$##$}\tabskip\displaywidth\crcr
    #1\crcr}}

\makeatother

\newdimen\pixel \pixel=.00333333 true in

\documentclass[11pt]{article}
\usepackage{epsfig}
\usepackage{fullpage}

\makeatletter
\def\bigpar{\bigbreak\@afterindentfalse\@afterheading\ignorespaces}
\def\medpar{\medbreak\@afterindentfalse\@afterheading\ignorespaces}
\def\smallpar{\smallbreak\@afterindentfalse\@afterheading\ignorespaces}

\newlength{\saveindent}
\newenvironment{proof}%
      {\bigpar{\bf Proof:}\ 
             \setlength{\saveindent}{\parindent} 
                       \ignorespaces}%
      {\stopproof\ignorespaces\bigbreak \setlength{\parindent}{\saveindent}}

      {\bigpar{\bf Proof:}%
             \setlength{\saveindent}{\parindent} 
                       \,\ignorespaces}%
      {\ignorespaces\bigbreak \setlength{\parindent}{\saveindent}}

      {\bigpar{\bf Proof:}\ %
             \setlength{\saveindent}{\parindent} 
                       \ignorespaces}%
      {\ignorespaces\bigbreak \setlength{\parindent}{\saveindent}}

\newenvironment{proofof}[1]%
      {\bigpar{\bf#1:}\ %
             \setlength{\saveindent}{\parindent} 
                       \ignorespaces}%
      {\stopproof\ignorespaces\bigbreak \setlength{\parindent}{\saveindent}}

\newenvironment{remark}%
      {\smallpar{\bf Remark:}\ 
                       \ignorespaces}%
      {\stopproof\ignorespaces\medbreak \setlength{\parindent}{\saveindent}}

\newenvironment{remark*}%
      {\smallpar{\bf Remark:}\ 
                       \ignorespaces}%
      {\ignorespaces\medbreak \setlength{\parindent}{\saveindent}}

      {\smallpar{\bf Remarks:}\ 
                       \ignorespaces}%
      {\stopproof\ignorespaces\medbreak \setlength{\parindent}{\saveindent}}

\newenvironment{remarks*}%
      {\smallpar{\bf Remarks:}\ 
                       \ignorespaces}%
      {\ignorespaces\medbreak \setlength{\parindent}{\saveindent}}

      {\smallpar{\bf Remark:}\ %
                       \ignorespaces}%
      {\stopproof\ignorespaces\medbreak \setlength{\parindent}{\saveindent}}

      {\smallpar{\bf Remarks:}\ %
                       \ignorespaces}%
      {\stopproof\ignorespaces\medbreak \setlength{\parindent}{\saveindent}}
\makeatother

\newtheorem{theorem}{Theorem}
\newtheorem{lemma}[theorem]{Lemma}

\newtheorem{proposition}[theorem]{Proposition}

\newtheorem{definition}[theorem]{Definition}
\newtheorem{example}{Example}

\def\begex{\begin{example}\parindent=0pt \rm}
\def\endex{\end{example}}
\def\square{\vbox{\hrule height.2pt\hbox{\vrule width.2pt height5pt \kern5pt
                                   \vrule width.2pt} \hrule height.2pt}}
\def\stopproof{\hfill \square \smallskip}

\def\half{{\textstyle{1\over2}}}

\def\quarter{{\textstyle{1\over4}}}

\def\eighth{{\textstyle{1\over8}}}

\def \xt {{\tilde X}}

\def \zt {{\tilde Z}}

\def \r {{\bf R}}

\def \r {{ \cal F}}

\def\r|{{\Bigr\vert}}
\def\l|{{\Bigl\vert}}

\def\phi {\Phi}

\def\e{\epsilon}

\def \zt {{\tilde Z}}

\def\varepsilon{\mathchar"122 }

\def \chi {{\mathbf 1}}

\def\law{{\cal L}}

\def\P {{ \bf P}}

\def\e {{ \bf {E}}}

\def \id {{\rm id}}

\def\Sq{{\cal S}_q}
\def\Sq-{{\cal S}_{q-1}}

\def\u{{\cal U   }}

\def \S {{\cal S}}

\def\f {{\cal F}}
\def \fh {{\widehat \f}}
\def\g {{\cal G}_{j+1}}

\def\given {{\,|\,}}

\def\ent{{ \rm \sc ENT}}

\newcommand{\be}{\begin{equation}}
\newcommand{\ee}{\end{equation}}
\newcommand{\lab}{\label}
\begin{document}
\title{Improved mixing time bounds for the \\ 
  Thorp shuffle }
\author{
{\sc Ben Morris}\thanks{Department of Mathematics,
University of California, Davis.
Email:
{\tt morris@math.ucdavis.edu}. 
Research partially supported 
NSF grant DMS-0707144.}
}  
\date{}
\maketitle
\begin{abstract}
\noindent 
E.~Thorp  
introduced the following card
shuffling model. 
Suppose the number of cards $n$ is even. 
Cut the deck into two equal piles. 
Drop the first card from the left pile or from the right pile
according 
to the outcome of a fair coin flip. Then drop from the other pile. 
Continue this way until both piles are empty. 
We show that if $n$ is  a power of $2$
then the mixing time of the Thorp shuffle is 
$O(\log^3 n)$. Previously, the best known bound 
was $O(\log^{4} n)$.
\end{abstract}
\setcounter{page}{1}

\section{Introduction} \lab{intro}
Card shuffling has a rich history in mathematics, dating back to 
work of Markov \cite{markov} and Poincare
\cite{poincare}. A basic problem is to determine 
the mixing time, i.e., the number of shuffles necessary to mix up the deck
(sec Section \ref{secaps} for a precise definition).
In \cite{imp}, the author found 
a general method 
that reduces bounding the mixing time of a card shuffle to
verifying a local
condition that involves only pairs of cards. 
This was used to give mixing time bounds for the Thorp shuffle 
and Durrett's $L$-reversal chain. In the present paper, 
we build on the techniques of \cite{imp}
and get an improved analysis of the Thorp shuffle. 

\section{Previous work}
Thorp \cite{th} introduced the following card shuffling model in 1973. 
Assume that the number of cards, $n$, is even.
Cut the deck into two equal piles. 
Drop the first card from the left pile or the
right
pile according to the outcome of a fair coin flip;  
then drop from the other pile. 
Continue this way, 
with independent coin flips deciding whether to drop {\sc left-right} or 
{\sc right-left} each time, until 
both piles are empty. 

Analyzing the Thorp shuffle 
is an old problem with theoretical roots. 
However, recently the Thorp shuffle has 
found applications in applied cryptography. The author,
Phil Rogaway and Till Stegers have used the Thorp shuffle as 
the basis for a practical algorithm for encoding small messages
such as social security numbers and credit card numbers
(see \cite{cthorp}). In order to analyze the algorithm
it is important to have good bounds on the mixing time. 

 The  Thorp shuffle, despite its simple description, 
has been hard
to analyze. Determining its mixing 
time 
has been called 
the ``longest-standing open card shuffling problem'' \cite{per}. 
In \cite{thorp} the author obtained the first 
poly log upper bound, proving a 
bound of $O(\log^{44} n)$, 
valid when $n$ is a power of $2$.  
Montenegro and Tetali \cite{mt-thorp} built on this
to get a bound of $O(\log^{29} n)$. 
In \cite{imp} the bound 
was improved to $O(\log^4 n)$, with no 
power-of-two assumption.
In the present paper we show that if the number of cards is 
a power of two, then the mixing time is 
$O(\log^3 n)$.

\section{Background} 
\label{bk}
\label{secaps}
 In this section we give some basic definitions and recall some notation
from \cite{imp}.  
  Let $p(x,y)$ 
be transition probabilities for a   Markov chain on a finite state space $V$
with a uniform stationary distribution.
For probability measures $\mu$ and $\nu$ on $V$,
define the total variation distance
$|| \mu - \nu || = \sum_{x \in V} |\mu(x) - \nu(x) |$, 
and 
define the mixing time
\be
\label{mixingtime}
T_{\rm mix} = \min \{n: || p^n(x, \, \cdot) - \u || \leq \quarter \mbox{ for all $x \in V$}\} \,,
\ee
where $\u$ denotes the uniform distribution.

 For a
probability distribution $\{p_i: i \in V\}$, define the 
(relative) entropy of $p$ by 
$\ent(p) = \sum_{i \in V} p_i \log (|V| p_i)$,
where we define $0 \log 0 = 0$.
The following well-known inequality links relative entropy to total 
variation distance. We have 
\begin{equation}
\label{totent}
|| p - \u || \leq \sqrt{ \half \ent(p)}.
\end{equation}
If $X$
is a random variable (or random permutation) 
taking finitely many values, 
define $\ent(X)$ as the relative entropy of the
distribution of $X$.
Note that if $\P(X = i) = p_i$  for $i \in V$ then
$\ent(X) = \e(\log (|V| p_X))$. 
We
shall think
of the distribution 
of a random permutation
in $\S_n$ as a sequence of probabilities of length $n!$, 
indexed by permutations in $\S_n$.
If $\f$ is a sigma-field, then we shall write 
$\ent(X \given \f)$ for the relative entropy of the conditional
 distribution of $X$ given $\f$. 
Note that $\ent(X \given \f)$ is a random variable.
If $\pi$ is a random permutation
in $S_n$, 
then for $1 \leq k \leq n$,  
define
$\f_k = \sigma( \pi^{-1}(k), \dots, \pi^{-1}(n))$, 
and 
define $\ent(\pi, k) = \ent( \pi^{-1}(k) \given \f_{k+1})$
(where we think of the 
conditional distribution 
of $\pi^{-1}(k)$ given $\f_{k+1}$ 
as being a sequence of length
$k$). 
The standard entropy chain rule (see, e.g., \cite{cover})
gives the following proposition.
\begin{proposition}
\label{decomp}
For any $i \leq n$ we have
\[
\ent(\pi) = \e \Bigl(\ent( \pi \given \f_{i})\Bigr) 
+ \sum_{k=i}^n \e(\ent(\pi, k)).
\]
\end{proposition}
To compute the relative entropy in first term on the right hand side,
we think of the distribution of $\pi$ given $\f_i$ as a sequence of
probabilities of length $(i-1)!$. 
\begin{remark}
Substituting $i = 1$ into the formula 
gives $\ent(\pi) = \sum_{k=1}^n \e(\ent(\pi,k))$.
\end{remark} 
If we think of $\pi$ as representing the order of a deck of cards,
with $\pi(i) = \mbox{location of card $i$}$, then this allows us to 
think of $\e(\ent(\pi, k))$ as the portion of the overall entropy
$\ent(\pi)$
that is attributable to the location $k$. 
We will also need the following proposition.
\begin{proposition}
\label{comps}
Let $\nu_1$ and $\nu_2$ be random permutations
on $\{0,\dots, n-1\}$. Suppose that there
is a set $W \subset \{0,1,\dots, n-1\}$ such that
$\nu_1^{-1}(x) = \nu_2^{-1}(x)$ for all $x \in W$. Let 
$\f = \sigma( \nu_1^{-1}(x): x \in W)$. Then 
\[
\ent(\nu_1) - \ent(\nu_2) = 
\e(\ent(\nu_1 \given \f) - \ent(\nu_2 \given \f)).
\]
\end{proposition}
\begin{proof} 
By the chain rule for entropy, for $i=1,2$ we can write 
\[
\ent(\nu_i) 
=
\ent(\nu_i^{-1}(x): x \in W) + 
 \e(\ent(\nu_i \given \f) 
\]
Since the first term doesn't depend on $i$ the proposition follows. 
\end{proof}
\begin{definition}
For $p, q \geq 0$, define $d(p, q) = \half p\log p + \half q \log q - 
{p + q \over 2} \log\Bigl({p + q \over 2} \Bigr)$.
\end{definition}
We will need the following proposition, which is easily 
verified using calculus.
\begin{proposition}(\cite{imp})
\label{conv}
Fix $p \geq 0$. 
The function $d(p,\,\cdot)$ is convex.
\end{proposition}

Observe that $d(p,q) \geq 0$, with equality iff $p=q$ by the 
strict convexity
of the function $x \to x \log x$. 
If $p = \{p_i: i \in V\}$ and $q = \{q_i: i \in V\}$ 
are both probability distributions on $V$, then we can define the 
``distance'' $d(p,q)$ between $p$ and $q$, by $d(p,q) = \sum_{i \in V}
d(p_i, q_i)$. 
(We use the term {\it distance} loosely and don't claim that $d(\cdot,\,
\cdot)$ satisfies the triangle inequality.)
Note that $d(p,q)$ is the difference between the average
of the entropies of $p$ and $q$ and the entropy of the average
(i.e. an even mixture) of $p$ and $q$. 

We will use the following projection lemma.
\begin{lemma}(\cite{imp})
\label{projection}
Let $X$ and $Y$ be random variables with distributions $p$ and $q$,
respectively. Fix a function $g$ and let $P$ and $Q$ be the distributions of $g(X)$ and
$g(Y)$, respectively. Then $d(p, q) \geq d(P,Q)$. 
\end{lemma}
Let $\u$ denote the uniform distribution on $V$. Note that if $\mu$ is 
an arbitrary distribution on $V$, then $\ent(\mu)$ and 
$d(\mu, \u)$ are both notions of a distance from $\mu$ to $\u$.
The following lemma relates the two.
\begin{lemma}(\cite{imp})
\label{dent}
For any distribution $\mu$ on $V$ we have
\[
d(\mu, \u) \geq {c \over \log |V|} \ent(\mu),
\]
for a universal constant $c > 0$.
\end{lemma}

A card shuffle can be described as 
a
random permutation chosen 
from 
a certain 
probability distribution.
If we start with 
the identity permutation and
each shuffle has the distribution of
$\pi$, then after $t$ steps the cards are distributed
like
$\pi_1 \cdots \pi_t$,
where the $\pi_i$ are i.i.d.~copies of $\pi$. 

\section{Thorp shuffle}
\label{secthorp}

Recall that the Thorp shuffle has the following description. 
Assume that the number of cards, $n$, is even.
Cut the deck into two equal piles. 
Drop the first card from the left pile or the
right
pile according to the outcome of a fair coin flip;  
then drop from the other pile. 
Continue this way, 
with independent coin flips deciding whether to drop {\sc left-right} or 
{\sc right-left} each time, until 
both piles are empty. 

We will actually work with the time reversal of the Thorp shuffle,
which has the same mixing time (since the Thorp shuffle is a 
random walk on a group; see \cite{rwg}). 
For convenience, we assume that $n = 2^d$ is a power of two. 
By writing the position of each card, from the bottom card ($0$)
to the top card ($2^{d} - 1$), in binary, we can view the
positions as elements 
of
the
$d$-dimensional unit hypercube $\{0,1\}^d$.
The reverse Thorp (RT) shuffle can then be constructed in the 
following way (see, e.g., \cite{cthorp}). 
Let $Z = \Bigl\{Z(l, t): l \in \{0,1\}^{d-1}, t \in \{0,1, \dots\}\Bigr\}$
be a collection of  i.i.d., Bernoulli(1/2) random variables. 
Note that $x \in \{0,1\}^d$ can be written as 
$x = (L(x), R(x))$, where $L(x)$ and $R(x)$ are the leftmost $d-1$ and 
rightmost bit, respectively, of $x$. The transition rule for the RT 
shuffle is as follows. At time $t$, suppose that the current state
$X_t = \pi$. Then the new state $X_{t+1} = \nu \circ \pi$, 
where $\nu$ is the permutation that sends
\[
(L, R) \to (R \oplus Z(L, t), L).
\]
We are now ready to state the technical result of this paper.
\begin{lemma}
\label{mainlemma}
Let $X_t$ be the reverse Thorp shuffle with $2^d$ cards. 
There is a universal constant $c$ such that 
if
$\mu$ is a random permutation which is independent of $\{X_t\}$
then
\[
\ent( X_d \circ \mu) \leq (1 - c/d) \ent(\mu).
\]
\end{lemma}
Before proving this lemma we show how it gives the desired mixing
time bound.
\begin{theorem}
The mixing time of the reverse Thorp shuffle with $2^d$ cards
is $O(d^3)$.
\end{theorem}
\begin{proof}
Repeated applications of Lemma \ref{mainlemma}
give 
\begin{eqnarray*}
\ent(X_{kd}) &\leq& (1 - c/d)^k 
\ent(\id) \\
&\leq& e^{-ck/d} d 2^d.
\end{eqnarray*}
Now let $\alpha$ be large enough so that 
$\left( {2 \over e^\alpha} \right)^d \leq 1/8$ for all $d$. Then
if $k = \lceil \alpha d^2 /c \rceil$ we have 
\begin{eqnarray*}
\ent(X_{kd}) \leq e^{-ck/d} d 2^d \leq \eighth
\end{eqnarray*}
and hence $|| X_{kd} - \u || \leq \quarter$
by equation \ref{totent}. The theorem 
follows since $k$ is $O(d^2)$. 
\end{proof}
We now give the proof 
of lemma \ref{mainlemma}.
\begin{proofof}{Proof of Lemma \ref{mainlemma}}
Fix an integer $T \geq 1$. For integers $j < n$, 
define $T_j = \lfloor \log_2 j \rfloor + 1 - T$.
Note that $T_0 \leq T_1 \leq \cdots T_{n-1}$. 
Let $\zt$ be obtained from $Z$ by flipping the value 
of $Z(L(X_{T_j}), T_j)$ for all $j$. More precisely, 
define
\[
\zt(l, t) =
\left\{\begin{array}{ll}
1 - Z(l, t) &
\mbox{if for some $j$ we have $L(X_t(j)) = l$ and $T_j = t$;} \\
Z(l, t) & \mbox{otherwise.}\\
\end{array}
\right.
\]
Let $\{\xt_t: t \geq 0\}$ be the reverse Thorp shuffle process 
defined by using $\zt$ instead of $Z$. 
For $j$ with $0 \leq j < n$, 
define
$\Gamma_j(X) = (X_1(j), \dots X_d(j))$, 
with a similar definition for $\Gamma_j(\xt)$.
For $k$ with $0 \leq k \leq n$, define
\begin{eqnarray*}
\f_k &=& \sigma( \Gamma_j(X), \Gamma_j(\xt): j \geq k) \\
&=& \sigma( X_t(j), \xt_t(j): j \geq k, 0 \leq t\leq d)
\end{eqnarray*}
Since $\f_n$ is trivial and $X_d$ is $\f_0$-measurable, 
we have
\begin{eqnarray}
\ent(X_d \circ \mu) - \ent(\mu) &=& \ent(X_d \circ \mu \given
\f_n) - \ent(X_d \circ \mu \given \f_0) \\
&=& 
\label{colsum}
\sum_{j = 0}^{n-1} 
\ent(X_d \circ \mu \given \f_{j+1}) - 
\ent(X_d \circ \mu \given \f_{j})
\end{eqnarray} 
We claim that for all $j$ with $0 \leq j < n$
we have
\begin{equation}
\label{mainclaim}
\e\Bigl( \ent(X_d \circ \mu \given \f_{j+1}) - 
   \ent(X_d \circ \mu \given \f_{j}) \Bigr) 
\leq -\ent(\mu, j) {c \over d},
\end{equation}
where $c > 0$ is a universal constant. 
Note that combining this with equation (\ref{colsum})
gives 
\begin{eqnarray}
\ent(X_d \circ \mu) - \ent(\mu) &\leq& 
{c \over d}
\sum_{j = 0}^{n-1} 
\leq \ent(\mu, j) 
= {c \over d} \ent(\mu),
\end{eqnarray} 
which proves the lemma. 
It remains to verify equation (\ref{mainclaim}). 

For $j$ with $0 \leq j \leq n$, define 
$\fh_j = \sigma( \f_{k+1}, \{ \Gamma_j(X), \Gamma_j(\xt) \}).$ 
Note that this is  the sigma field generated by $\f_{k+1}$ 
and the {\it unordered set} 
$\{ \Gamma_j(X), \Gamma_j(\xt) \}.$
Note that $\fh_j \supset \f_{j+1}$. Hence
for all $j$ with $0 \leq j < n$ we have 
\begin{equation}
\label{ineq}
\e(\ent(X_d \circ \mu \given \f_{j+1}))
\leq 
\e(\ent(X_d \circ \mu \given \fh_{j})),
\end{equation}
by Jensen's inquality applied to $x \to x \log x$.
Let $W = \{X_d(j+1), \dots, X_d(n-1)\}$
and let $\g$ denote the sigma-field generated by 
$(X_d \circ \mu)^{-1}(x)$ for $x \in W$. 
Let $\g' = \sigma( 
\mu^{-1}(j+1), \dots, \mu^{-1}(n-1))$.
Then
\begin{eqnarray}
\label{constuff}
\ent(X_d \circ \mu \given \fh_{j})
-
\ent(X_d \circ \mu \given \f_{j}) 
&=&
\e \Bigl(
\ent(X_d \circ \mu \given \fh_{j}, \g) \\
& & \qquad
-
\ent(X_d \circ \mu \given \f_{j}, \g) \Bigr) \\
&=&
\e \Bigl(
\ent(X_d \circ \mu \given \fh_{j}, \g') \\
& &\qquad
-
\ent(X_d \circ \mu \given \f_{j}, \g') \Bigr), 
\end{eqnarray}
where the first equality holds by 
Proposition \ref{comps}.
Note that $\f_{j+1} = \sigma(S, Z(l,t): (l, t) \in S),$
where 
\[
S = \{ (r,t): \mbox{
$L(X_t(i)) = l$ or 
$L(\xt_t(i)) = l$ for some $i > j$}\},
\] 
that is, $S$ is the collection of bits used to generate 
$\Gamma_i(X)$ and $\Gamma_i(\xt)$ for $i > j$.

We shall refer to indices $i$ with $0 \leq i < n$ as
{\it cards.}
Say that cards $i$ and $j$ are {\it adjacent at time $t$}
if $L(X_t(i)) = L(X_t(j))$. If $T_j \geq 0$, let
$m(j)$ be the card adjacent to $j$ at time $T_j$. 

 Note that $\f_j = \sigma(\fh_{j+1}, Z( L(X_{T_j}(j)), T_j))$. 
Therefore, on the event that 
$(L(X_{T_j}(j)), T_j) \in S$
the expression on the lefthand-side of (\ref{constuff})
is $0$. However, we now show that if $m(j) < j$, then 
$(L(X_{T_j}(j)), T_j) \notin S$. 

Note that if $(l ,t) \in S$, then either $L(X_t(i)) = l$ 
for some $i > j$, or $t > T_i$ for some $i > j$
(and hence $t > T_j$). Thus if $m(j) < j$, then
$(L(X_{T_j}(j)), T_j) \notin S$. So on the event that 
$m(j) < j$ and 
$\{ \Gamma_j(X), \Gamma_j(\xt) \} = \{\Gamma, \Gamma'\},$
the conditional distribution of 
$(\Gamma_j(X), \Gamma_j(\xt))$
given
$\fh_{j}$ 
is an even mixture of 
$(\Gamma, \Gamma')$ and 
$(\Gamma', \Gamma)$,
according to the value of $Z(L(X_{T_j}(j), T_j)$.

Let $\law(W \given \f)$ denote the conditional 
distribution of random variable (or random
permutation)  $W$ given the sigma field $\f$.
Note that
\begin{eqnarray*}
\law(X_d \circ \mu \given \fh_{j}, \g) &=&
\half \law(X_d \circ \mu \given \f_j, \g) +
\half \law(\xt \circ \mu \given \f_j, \g). 
\end{eqnarray*}
Therefore, 
\begin{eqnarray*}
& & \ent(X_d \circ \mu \given \fh_{j+1}, \g) \\
&\leq&
\half \ent(X_d \circ \mu \given \f_j, \g) +
\half \ent(\xt \circ \mu \given \f_j, \g) - 
d(\law(X_d \circ \mu \given \f_j),\law(\xt \circ \mu \given \f_j, \g)) \\
&=&  \ent(X_d \circ \mu \given \f_j, \g) - d( 
\law(X_d \circ \mu \given \f_j, \g),
\law(\xt \circ \mu \given \f_j, \g)).
\end{eqnarray*}
But by the projection lemma,
\begin{eqnarray*}
d(
\law(X_d \circ \mu \given \f_j, \g), 
\law(\xt_d \circ \mu \given \f_j, \g))
&\geq&  
d(
\law( (X_d \circ \mu)^{-1}(X_d(j)) \given \f_j, \g), 
\law( (\xt_d \circ \mu)^{-1}(X_d(j)) \given \f_j, \g)) \\
&=&  
d(
\law( \mu^{-1}(j) \given \f_j, \g'), 
\law( \mu^{-1}(m(j)) \given \f_j, \g')). 
\end{eqnarray*}
Since $\mu$ is independent of $X_d$, this last quantity 
is $d( 
\law( \mu^{-1}(j) \given \g' ), 
\law( \mu^{-1}(m(j) \given \g'    )$. Combining this with equation 
(\ref{constuff}) gives
\begin{equation}
\label{finn}
\e \Bigl( \ent(X_d \circ \mu \given \fh_{j})
-
\ent(X_d \circ \mu \given \f_{j})  \Bigr)
\leq
-\e \Bigl(
d( 
\law( \mu^{-1}(j) \given \g' ), 
\law( \mu^{-1}(m(j)) \given \g'    ) \Bigr).
\end{equation}
Let $j_{d-1}j_{d-2}\cdots j_0$ be the binary representation of $j$.
For cards $k$ and $j$, write $D(k,j) = \max\{i: k_i \neq j_i\}$. 
Note that $D(k, j)$ is the minimum value of $t$ such that 
there is positive probability that
$k$ and $j$ are adjacent after $t$ steps. 
For $t \geq 0$, let $B(j, t) = \{k: D(k,j) = t\}$.
For convenience, let $B(j, t) = \emptyset$ if $t < 0$. 
Let $I = \{0, 1, \dots, j-1\}$.
Note that if $k \in B(j, T_j) \cap I$, then
$\P( m(j) = k) = \left({1 \over 2}\right)^{T_j}$. Equation (\ref{finn})
implies that
\begin{eqnarray*}
& & \e \Bigl( \ent(X_d \circ \mu \given \fh_{j})
-
\ent(X_d \circ \mu \given \f_{j})  \Bigr) \\
&\leq&
-\sum_{k \leq j} \P(m(j) = k)
\e\Bigl( d( 
\law( \mu^{-1}(j) \given \g' ), 
\law( \mu^{-1}(m(j)) \given \g')  \Bigr)   \\
&=& -\sum_{k \in B(j, T_j) \cap I}  
\left({1 \over 2}\right)^{T_j}
\e\Bigl( d( 
\law( \mu^{-1}(j) \given \g' ), 
\law( \mu^{-1}(m(j)) \given \g')  \Bigr) \\
&=& - \sum_{k \in B(j, T_j) \cap I}  
\left({1 \over 2}\right)^{r + 1 - T}
\e \Bigl( d( 
\law( \mu^{-1}(j) \given \g' ), 
\law( \mu^{-1}(m(j)) \given \g') \Bigr),
\end{eqnarray*}
where $r = \lceil \log_2 j \rceil$. It follows that if $T$ is 
a {\it random variable}  and $T_j = r + 1 - T$, then 
\begin{eqnarray*}
& & \e \Bigl( \ent(X_d \circ \mu \given \fh_{j})
-
\ent(X_d \circ \mu \given \f_{j})  \Bigr) \\
&\leq&
- \e \Bigl( \sum_{k \in B(j, T_j) \cap I}  
\left({1 \over 2}\right)^{r + 1 - T}
d( 
\law( \mu^{-1}(j) \given \g' ), 
\law( \mu^{-1}(m(j) \given \g') \Bigr). 
\end{eqnarray*}
In particular, if $T$ is geometric($1/2$), we have 
\begin{eqnarray*}
& & \e \Bigl( \ent(X_d \circ \mu \given \fh_{j+1})
-
\ent(X_d \circ \mu \given \f_{j})  \Bigr) \\
&\leq&
-\sum_{t=1}^\infty \left( {1 \over 2} \right)^t 
\left( { 1 \over 2} \right)^{r + 1 - t}
\sum_{k \in B(j, r + 1 - t) \cap I}  
d\Bigl( 
\law( \mu^{-1}(j) \given \g' ), 
\law( \mu^{-1}(m(j) \given \g') \Bigr) \\ 
&=&
-\left( { 1 \over 2} \right)^{r+1}
\sum_{k \leq j}  
d\Bigl( 
\law( \mu^{-1}(j) \given \g' ), 
\law( \mu^{-1}(m(j) \given \g') \Bigr). 
\end{eqnarray*}
Since $j \geq 2^{r-1}$, this is at most
\begin{eqnarray*}
& & 
-{1 \over 4} \Bigl(\, {1 \over j}
\sum_{k \leq j}  
d( 
\law( \mu^{-1}(j) \given \g' ), 
\law( \mu^{-1}(m(j) \given \g') \Bigr) \\
&\leq&
-{1 \over 4} \Bigl(\, 
d( 
\law( \mu^{-1}(j) \given \g' ), 
{1 \over j}
\sum_{k \in I}  
\law( \mu^{-1}(m(j) \given \g') \Bigr) \\
&\leq&
-c \ent(\mu, j), 
\end{eqnarray*}
for a universal constant $c$, 
where the first inequality follows from 
Proposition \ref{conv} and the second inequality
follows from Proposition \ref{dent}
(since the second argument of $d$ is the uniform distribution). 
Combining this with equation (\ref{ineq}) verifies 
equation (\ref{mainclaim}), which completes the 
proof.
\end{proofof}
The above analyis extends to the non 
power-of-two case and we intend to handle this in the final 
version of this paper. 

\end{document}